\documentclass[12pt]{amsart}
\usepackage{amsmath}
\usepackage{amsthm}
\usepackage{amssymb}
\usepackage{amscd}
\usepackage{amsfonts}
\usepackage{amsbsy}
\usepackage{epsfig}
\usepackage{graphicx}
\usepackage{tikz}
\usepackage{tikz-cd}
\usetikzlibrary{arrows}
\usepackage{comment}
\usepackage{stmaryrd}
\usepackage{multicol}

\usetikzlibrary{datavisualization}
\usetikzlibrary{datavisualization.formats.functions}

\newtheorem{theorem}{Theorem}[section]
\newtheorem{corollary}[theorem]{Corollary}
\newtheorem{proposition}[theorem]{Proposition}
\newtheorem{lemma}[theorem]{Lemma}

\theoremstyle{definition}
\newtheorem{remark}[theorem]{Remark}

\newtheorem{definition}[theorem]{Definition}
\newtheorem{example}[theorem]{Example}

\def\ie{{\em i.e.,} }
\def\eg{{\em e.g.} }

\newfont\bbf{msbm10 at 12pt}
\def\eps{\varepsilon}
\def\phi{\varphi}
\def\R{{\mathbb R}}

\def\N{{\mathbb N}}

\def\diam{\mbox{\rm diam}\,}

\def\theta{\vartheta}

\def\chain{{\mathcal C}}

\def\eps{\varepsilon}

\begin{document}

\title{Planar embeddings of Minc's continuum and generalizations}

\author{Ana Anu\v{s}i\'c}
\address[A.\ Anu\v{s}i\'c]{Departamento de Matem\'atica Aplicada, IME-USP, Rua de Mat\~ao 1010, Cidade Universit\'aria, 05508-090 S\~ao Paulo SP, Brazil}
\email{anaanusic@ime.usp.br}
\thanks{We thank Henk Bruin and Jernej \v Cin\v c for helpful remarks on the earlier versions of this paper. AA was supported by grant 2018/17585-5, S\~ao Paulo Research Foundation (FAPESP)}
\date{\today}

\subjclass[2010]{54C25, 54F15, 54H20, 37B45, 37C70, 37E05}
\keywords{planar embeddings, accessible points, chainable continuum, attractor, inverse limit space}

\maketitle

\begin{abstract}
	We show that if $f\colon I\to I$ is piecewise monotone, post-critically finite, and locally eventually onto, then for every point $x\in X=\underleftarrow{\lim}(I,f)$ there exists a planar embedding of $X$ such that $x$ is accessible. In particular, every point $x$ in Minc's continuum $X_M$ from \cite[Question 19 p. 335]{problems} can be embedded accessibly. All constructed embeddings are {\em thin}, \ie can be covered by an arbitrary small chain of open sets which are connected in the plane.
\end{abstract}

\section{Introduction}

The main motivation for this study is the following long-standing open problem:

{\bf Problem}\hspace{-2pt} (Nadler and Quinn 1972 \cite[p. 229]{Nadler} and \cite{NaQu}){\bf .} Let $X$ be a chainable continuum, and $x\in X$. Is there a planar embedding of $X$ such that $x$ is accessible?

The importance of this problem is illustrated by the fact that it appears at three independent places in the collection of open problems in Continuum Theory published in 2018 \cite[see Question 1, Question 49, and Question 51]{problems2018}. We will give a positive answer to the Nadler-Quinn problem for every point in a wide class of chainable continua, which includes $\underleftarrow{\lim}(I,f)$ for a simplicial locally eventually onto map $f$, and in particular continuum $X_M$ introduced by Piotr Minc in \cite[Question 19 p. 335]{problems}. Continuum $X_M$ was suspected to have a point which is inaccessible in every planar embedding of $X_M$. 

A continuum is a non-empty, compact, connected, metric space, and it  is chainable if it can be represented as an inverse limit with bonding maps $f_i\colon I\to I$, $i\in\N$, which can be assumed to be onto and piecewise linear. That is,
$$X=\underleftarrow{\lim}(I,f_i)=\{(\xi_0,\xi_1,\xi_2,\ldots): f_i(\xi_i)=\xi_{i-1}, i\in\N\}\subset I^{\infty},$$
where $I=[0,1]$ and $I^{\infty}$ is equipped with the standard product topology. 

If $X\subset \R^2$ is a planar continuum, we say that $x\in X$ is {\em accessible} (from the complement of $X$), if there exists an arc $A\subset\R^2$ such that $A\cap X=\{x\}$. According to an old result of Bing \cite{Bing}, every chainable continuum can be embedded in the plane, making some points accessible and possibly leaving some inaccessible. In fact, if $X$ is indecomposable, there are going to be many inaccessible points in every planar embedding of $X$, see \cite{Maz}. For further results on planar embeddings of chainable continua and accessibility, see \eg the related results on the pseudo-arc in \cite{Bre,Lew, Sm},  unimodal inverse limit spaces in \cite{embed,AC}, Knaster continua in \cite{DT,MayThesis,May}, or hereditary decomposable chainable continua in \cite{MincTrans,Ozbolt}.	

In \cite{ABC-XS} (jointly with Henk Bruin and Jernej \v Cin\v c), we show that if $x=(x_0,x_1,x_2,\ldots)\in X=\underleftarrow{\lim}(I,f_i)$ is such that $x_i$ is {\em not in a zigzag of $f_i$} for every $i\in\N$, then we can embed $X$ in the plane with $x$ accessible, see \cite[Theorem 7.3]{ABC-XS}. Precise definition of what it means to be contained in a zigzag is given in Definition~\ref{def:zigzag}, see also Figure~\ref{fig:zigzag}, and the rest of Section~\ref{sec:main2} for some basic properties. Here we will shortly give an intuitive reason why this notion has an affect on accessibility. If $x$ is not in a zigzag of $f$, then we can ``permute" the graph of $f$ such that we ``expose" the point $(x,f(x))$ in the graph of $f$. To be more precise, for every $\eps>0$ there exists an embedding $\alpha\colon I\to I^2$, such that $|\pi_2(\alpha(y))-f(y)|<\eps$ for every $y\in I$, and the straight line joining $\alpha(x)$ with $(1,f(x))$ intersects $\alpha(I)$ only in $\alpha(x)$, see Figure~\ref{fig:exposing}. Here $\pi_2\colon I^2\to I$ is the projection on the second coordinate.

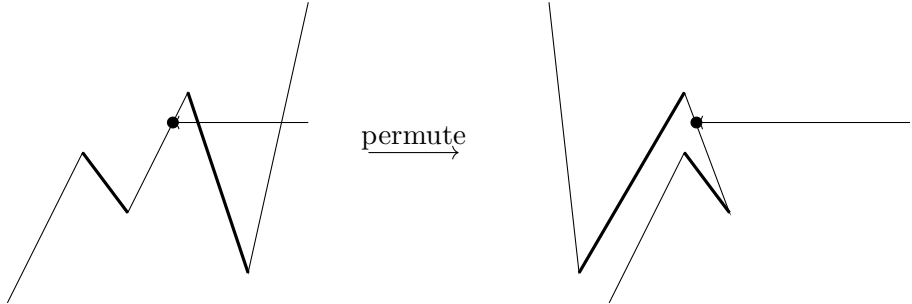
\begin{figure}[!ht]
	\centering
	\begin{tikzpicture}[scale=4]
	\draw (0,0)--(0.25,0.5)--(0.4,0.3)--(0.6,0.7)--(0.8,0.1)--(1,1);
	\draw[solid, fill] (0.55,0.6) circle (0.018);
	\draw[->] (1,0.6)--(0.55,0.6);
	\draw[very thick] (0.25,0.5)--(0.4,0.3);
	\draw[very thick] (0.6,0.7)--(0.8,0.1);
	
	\draw[->] (1.2,0.5)--(1.5,0.5);
	\node at (1.35,0.55) {\small permute};
	
	\draw (2,0)--(2.25,0.5)--(2.4,0.3)--(2.25,0.7);
	\draw[<-] (2.29,0.6)--(3,0.6);
	\draw[solid, fill] (2.29,0.6) circle (0.018);
	\draw (2.25,0.7)--(1.9,0.1)--(1.8,1);
	\draw[very thick] (2.25,0.5)--(2.4,0.3);
	\draw[very thick] (2.25,0.7)--(1.9,0.1);
	\end{tikzpicture}
	\caption{Permuting the graph to expose points. Graph of $f$ is given on the left, with point $(x,f(x))$ denoted by a circle. The horizontal arc joining $(x,f(x))$ with $(1,f(x))$ intersects the graph multiple times. On the right we construct a permutation $\alpha\colon I\to I^2$ of the graph, exposing $\alpha(x)$. Note that $(x,f(x))$ in the boldface area cannot be exposed.}
	\label{fig:exposing}
\end{figure}

The results of this paper will mostly be stated for chainable continua which can be represented as inverse limits with a single bonding map $f\colon I\to I$. Not every chainable continuum is like that, see \eg \cite{Mah}. The reason for this restriction, other than simplicity of notation, is the dynamical nature of spaces $\underleftarrow{\lim}(I,f)$. It was shown by Barge and Martin in \cite{BaMa}, with the use of Brown's theorem \cite{Brown}, that every $\underleftarrow{\lim}(I,f)$ can be embedded in the plane as global attractor of a planar homeomorphism $F\colon\R^2\to\R^2$, which acts on the attractor as the {\em shift homeomorphism} given by $\sigma((x_0,x_1,x_2,\ldots))=(f(x_0),x_0,x_1,\ldots)$. It is still not completely clear which planar embeddings of $\underleftarrow{\lim}(I,f)$ allow $\sigma$ to be extended to a homeomorphism of the plane. This question was first asked by Boyland (in 2015) for unimodal maps $f$, see the discussion in Section~8 of \cite{AC}.

The ideas in this paper originated from the study of continuum $X_M=\underleftarrow{\lim}(I,f_M)$, where $f_M\colon I\to I$ is given in Figure~\ref{fig:Minc}. It was introduced by Piotr Minc in 2001, where he asks:

{\bf Question}\,(Minc \cite[Question 19 p. 335]{problems}){\bf .} Is there a planar embedding of $X_M$ such that $p=(1/2,1/2,\ldots)$ is accessible?

Note that $1/2$ is in a zigzag of $f$, so the theory from \cite{ABC-XS} does not help.
Actually, $1/2$ is in a zigzag of $f^n$ for every $n\in\N$, so it is not helpful if we represent $X_M$ as $\underleftarrow{\lim}(I,f_M^{n_i})$, where $(n_i)_{i\in\N}$ is any sequence of natural numbers. However, it turns out that there is another representation of $X_M$ in which coordinates of $p$ will not be in zigzags of bonding maps. We will construct a map $g\colon I\to I$ for which there is a homeomorphism $h\colon X_M\to\underleftarrow{\lim}(I,g)$ such that $h(p)=(1/2,1/2,1/2,\ldots)$, and such that $1/2$ is not in a zigzag of $g$, thus answering Minc's question in positive. See the graph of $g$ in Figure~\ref{fig:g}. Actually, in Section~\ref{sec:Minc} we show that every point of $X_M$ can be embedded accessibly, see Theorem~\ref{thm:minc}. We note that all the constructed embeddings are {\em thin}, \ie the planar representation can be covered with an arbitrary small chain of open and connected sets in the plane. 

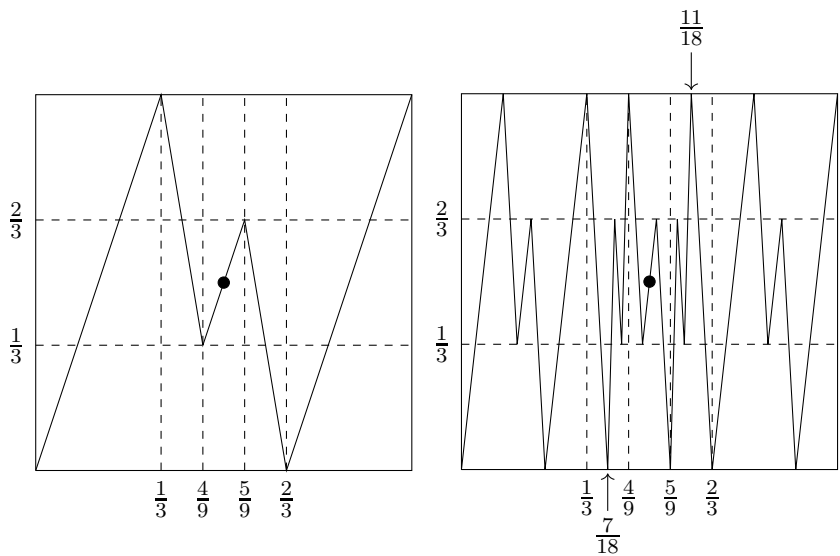
\begin{figure}[h!]
	\centering
	\vspace{-10pt}
	\begin{tikzpicture}[scale=5]
	\draw (0,0)--(0,1)--(1,1)--(1,0)--(0,0);
	\draw[dashed] (1/3,0)--(1/3,1);
	\draw[dashed] (2/3,0)--(2/3,1);
	\draw[dashed] (0,1/3)--(1,1/3);
	\draw[dashed] (0,2/3)--(1,2/3);
	\node[below] at (1/3,0) {\small $\frac 13$};
	\node[below] at (2/3,0) {\small $\frac 23$};
	\node[left] at (0,1/3) {\small $\frac 13$};
	\node[left] at (0,2/3) {\small $\frac 23$};
	\draw[dashed,thin] (4/9,0)--(4/9,1);
	\node[below] at (4/9,0) {\small $\frac 49$};
	\draw[dashed,thin] (5/9,0)--(5/9,1);
	\node[below] at (5/9,0) {\small $\frac 59$};
	
	\node[below] at (7/18,-0.175) {};
	\node[above] at (11/18,1.1) {};
	
	\draw[fill] (1/2,1/2) circle (0.015);
	
	\draw ( 0 , 0 ) --
	( 0.3333333333333333 , 1 ) --
	( 0.4444444444444444 , 0.3333333333333333 ) --
	( 0.5555555555555556 , 0.6666666666666666 ) --
	( 0.6666666666666666 , 0 ) --
	( 1 , 1 );
	
	\node[below] at (7/18,-0.19) {\small $ $};
	\end{tikzpicture}
	\begin{tikzpicture}[xscale=5,yscale=5]
	\draw (0,0)--(0,1)--(1,1)--(1,0)--(0,0);
	\draw[dashed] (1/3,0)--(1/3,1);
	\draw[dashed] (2/3,0)--(2/3,1);
	\draw[dashed] (0,1/3)--(1,1/3);
	\draw[dashed] (0,2/3)--(1,2/3);
	\node[below] at (1/3,0) {\small $\frac 13$};
	\node[below] at (2/3,0) {\small $\frac 23$};
	\node[left] at (0,1/3) {\small $\frac 13$};
	\node[left] at (0,2/3) {\small $\frac 23$};
	\draw[dashed,thin] (4/9,0)--(4/9,1);
	\node[below] at (4/9,0) {\small $\frac 49$};
	\draw[dashed,thin] (5/9,0)--(5/9,1);
	\node[below] at (5/9,0) {\small $\frac 59$};
	
	\draw[fill] (1/2,1/2) circle (0.015);
	
	\node[below] at (7/18,-0.1) {\small $\frac{7}{18}$};
	\draw[->] (7/18,-0.11)--(7/18,-0.01);
	\node[above] at (11/18,1.1) {\small $\frac{11}{18}$};
	\draw[->] (11/18,1.11)--(11/18,1.01);
	
	\draw ( 0 , 0.0 ) --
	( 0.1111111111111111 , 1.0 ) --
	( 0.14814814814814814 , 0.33333333333333326 ) --
	( 0.18518518518518517 , 0.6666666666666666 ) --
	( 0.2222222222222222 , 0.0 ) --
	( 0.3333333333333333 , 1.0 ) --
	( 0.3888888888888889 , 6.661338147750939e-16 ) --
	( 0.4074074074074074 , 0.6666666666666666 ) --
	( 0.42592592592592593 , 0.3333333333333346 ) --
	( 0.4444444444444444 , 0.9999999999999998 ) --
	( 0.48148148148148145 , 0.3333333333333336 ) --
	( 0.5185185185185185 , 0.666666666666666 ) --
	( 0.5555555555555556 , 0.0 ) --
	( 0.5740740740740741 , 0.6666666666666666 ) --
	( 0.5925925925925926 , 0.333333333333334 ) --
	( 0.6111111111111112 , 0.9999999999999989 ) --
	( 0.6666666666666666 , 0.0 ) --
	( 0.7777777777777778 , 0.9999999999999993 ) --
	( 0.8148148148148148 , 0.3333333333333336 ) --
	( 0.8518518518518519 , 0.6666666666666666 ) --
	( 0.8888888888888888 , 6.661338147750939e-16 ) --
	( 1 , 1.0 );
	\end{tikzpicture}
	\vspace{-10pt}
	\caption{Minc's map $f_M$ and its second iterate $f^2_M$, illustrating that $1/2$ is in a zigzag of $f_M^n$ for every $n\in\N$.}
	\label{fig:Minc}
\end{figure}

Finally, in Section~\ref{sec:general} we generalize the construction to $\underleftarrow{\lim}(I,f)$, where $f$ is assumed to be piecewise monotone, locally eventually onto (leo), and with eventually periodic critical points, see Corollary~\ref{cor:main}. The leo assumption is not very restrictive; any piecewise monotone interval map without restrictive intervals, periodic attractors, or wandering intervals is conjugate to a piecewise linear leo map, or semi-conjugate otherwise (see \eg \cite{dMvS}). Furthermore, every simplicial map $f$ has eventually periodic critical points. We note that Minc's map $f_M$ satisfies all the properties above. However, for the clarity of the exposition, we will explain the construction in the special case of $X_M$ before proceeding to the more general theory.

\section{Preliminaries}\label{sec:main}
Set of natural numbers will be denoted by $\N$ and $\N_0=\N\cup\{0\}$.  A {\em continuum} is a nonempty, compact, connected, metric space. An {\em arc} is a space homeomorphic to the unit interval $I=[0,1]$. Given two continua $X,Y$, a continuous function $f\colon X\to Y$ is called a {\em map}.  A map $f\colon I\to I$ is called {\em piecewise monotone} if there is $m\geq 0$, and points $0=c_0<c_1<\ldots<c_m<c_{m+1}=1$, such that $f|_{[c_i,c_{i+1}]}$ is strictly monotone for every $i\in\{0,\ldots,m\}$. For $i\in\{1,\ldots,m\}$, points $c_i$ are called {\em critical points of $f$}, and $\{c_1,\ldots,c_m\}$ is called {\em critical set of $f$}. For the simplicity of notation, we will often include $0=c_0$ and $c_{m+1}=1$ in the critical set.

Given a sequence of continua $X_i$, $i\in\N_0$, and maps $f_i\colon X_i\to X_{i-1}$, $i\in\N$, we define the {\em inverse limit space} of the inverse system $(X_i,f_i)$ as:
$$\underleftarrow{\lim}(X_i,f_i):=\{(\xi_0,\xi_1,\xi_2,\xi_3,\ldots): f_i(\xi_i)=\xi_{i-1}, i\in\N\}\subset \prod_{i=0}^{\infty} X_i,$$
and equip it with the product topology, \ie the smallest topology in which all coordinate projections $\pi_i\colon \underleftarrow{\lim}(X_i,f_i)\to X_i$, $i\in\N_0$ are continuous. Then $\underleftarrow{\lim}(X_i,f_i)$ is also a continuum. If there is a continuum $X$ such that $X_i=X$ for all $i\in\N$, the inverse limit space is denoted by $\underleftarrow{\lim}(X,f_i)$, and if additionally there is $f\colon X\to X$ such that $f_i=f$ for all $i\in\N$, it is denoted by $\underleftarrow{\lim}(X,f)$.

A {\em chain $\chain$} in a continuum $X$ is a set $\chain=\{\ell_1,\ldots,\ell_n\}$, where $\ell_i$, $i\in\{1,\ldots,n\}$ are non-empty open sets in $X$ such that $\ell_i\cap\ell_j\neq\emptyset$ if and only if $|i-j|\leq 1$. Sets $\ell_i$ are called {\em links} of $\chain$. Note that we do not necessarily assume that $\ell_i$ are connected sets in $X$ (and they most often will not be). {\em Mesh} of $\chain$ is the maximal diameter of all links of $\chain$. We say that $X$ is {\em chainable} if for every $\eps>0$ there is a chain in $X$ of mesh $<\eps$ which covers $X$. Every chainable continuum can be represented as $\underleftarrow{\lim}(I,f_i)$, for some maps $f_i\colon I\to I$ which can be assumed to be piecewise linear and surjective.

Given a map $f\colon I\to I$ and $n\in\N$, by $f^n$ we denote its $n$th iterate, \ie $f^1=f$ and $f^n=f^{n-1}\circ f$ for all $n>1$. Given a sequence of natural numbers $(n_i)_{i\in\N}$, the spaces $\underleftarrow{\lim}(I,f^{n_i})$ and $\underleftarrow{\lim}(I,f)$ are homeomorphic, with a homeomorphism given by $\underleftarrow{\lim}(I,f)\ni (\xi_0,\xi_1,\xi_2,\xi_3,\ldots)\mapsto(\xi_0,\xi_{n_1},\xi_{n_1+n_2},\xi_{n_1+n_2+n_3},\ldots)\in\underleftarrow{\lim}(I,f^{n_i})$. Moreover, if we are given a sequence $(n_i)_{i\in\N}$ of natural numbers, and maps $s_i,t_i,g_i\colon I\to I$ such that the diagram from Figure~\ref{fig:comm} commutes, then $\underleftarrow{\lim}(I,f)$ is homeomorphic to $\underleftarrow{\lim}(I,g_i)$. The homeomorphism between $\underleftarrow{\lim}(I,f^n_i)$ and $\underleftarrow{\lim}(I,g_i)$ is given by $(\xi_0,\xi_1,\xi_2,\xi_3,\ldots)\mapsto(s_1(\xi_1),s_2(\xi_2),s_3(\xi_3),\ldots)$. For the more general theory of homeomorphisms of inverse limits, see \cite{Miod}.

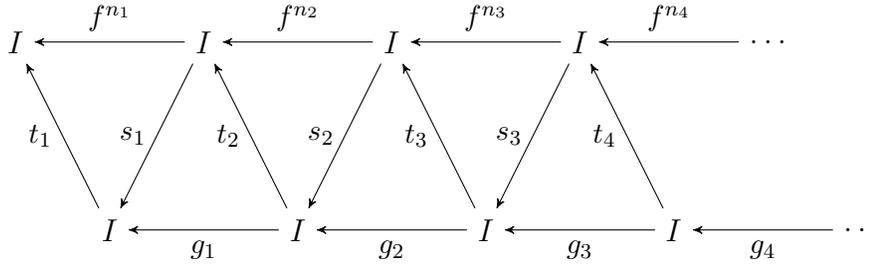
\begin{figure}[!ht]
	\centering
	\begin{tikzpicture}[->,>=stealth',auto, scale=2.5]
	\node (1) at (0,1) {$I$};
	\node (2) at (0.5,0) {$I$};
	\node (3) at (1,1) {$I$};
	\node (4) at (1.5,0) {$I$};
	\node (5) at (2,1) {$I$};
	\node (6) at (2.5,0) {$I$};
	\node (7) at (3,1) {$I$};
	\node (8) at (3.5,0) {$I$};
	\node (9) at (4,1) {$\ldots$};
	\node (10) at (4.5,0) {$\ldots$};
	\draw [->] (3) -- node[above]{\small $f^{n_1}$} (1);
	\draw [->] (5) -- node[above]{\small $f^{n_2}$} (3);
	\draw [->] (7) -- node[above]{\small $f^{n_3}$} (5);
	\draw [->] (9) -- node[above]{\small $f^{n_4}$} (7);
	\draw [->] (4) -- node[below]{\small $g_1$} (2);
	\draw [->] (6) -- node[below]{\small $g_2$} (4);
	\draw [->] (8) -- node[below]{\small $g_3$} (6);
	\draw [->] (10) -- node[below]{\small $g_4$} (8);
	\draw [->] (2) -- node[left]{\small $t_1$} (1);
	\draw [->] (3) -- node[left]{\small $s_1$} (2);
	\draw [->] (4) -- node[left]{\small $t_2$} (3);
	\draw [->] (5) -- node[left]{\small $s_2$} (4);
	\draw [->] (6) -- node[left]{\small $t_3$} (5);
	\draw [->] (7) -- node[left]{\small $s_3$} (6);
	\draw [->] (8) -- node[left]{\small $t_4$} (7);
	\end{tikzpicture}
	\caption{Commutative diagram.}
	\label{fig:comm}
\end{figure}

\section{Planar embeddings, zigzags, and accessibility}\label{sec:main2}

Given a continuum $X$, by its {\em planar embedding} we mean a homeomorphism $\nu\colon X\to \nu(X)\subset\R^2$. It is known that every chainable continuum can be embedded in the plane, \cite{Bing}. Given a continuum $X\subset\R^2$, and $\xi\in X$, we say that $\xi$ is {\em accessible from the complement} (or just {\em accessible}) if there is an arc $A\subset \R^2$ such that $A\cap X=\{\xi\}$. 

\begin{definition}\label{def:zigzag}
	Let $f: I\to I$ be a piecewise monotone map with 
	critical points $0<c_1< \ldots< c_{m}<1$. We say that $y\in I$ is 
	\emph{inside a zigzag of $f$} if for every $k\in\{1,\ldots,m-1\}$ such that $y\in [c_k,c_{k+1}]$, there exist $a,b\in I$
	such that $a<c_k<c_{k+1}<b\in I$ and either
	\begin{enumerate}
		\item $f(c_k)>f(c_{k+1})$ and $f|_{[a,b]}$ assumes its global minimum at $a$ and its global maximum
		at $b$, or
		\item $f(c_k)<f(c_{k+1})$ and $f|_{[a,b]}$ assumes its global maximum at $a$ and its 
		global minimum at $b$.
	\end{enumerate}
	See Figure~\ref{fig:zigzag}.
\end{definition}

\begin{figure}[!ht]
	\centering
	\begin{tikzpicture}[scale=4]
	\draw (0,0)--(1,0)--(1,1)--(0,1)--(0,0);
	\draw (0,0)--(1/5,1)--(2/5,1/4)--(3/5,3/4)--(4/5,1/2)--(1,1);
	\draw (2/5,-0.02)--(2/5,0.02);
	\node at (0.5, 0.8) { $f$};
	\node at (2/5, -0.07) {\scriptsize $a$};
	\draw (3/5,-0.02)--(3/5,0.02);
	\node at (3/5, -0.07) {\scriptsize $c_3$};
	\draw (4/5,-0.02)--(4/5,0.02);
	\node at (4/5, -0.07) {\scriptsize $c_4$};
	\draw (1,-0.02)--(1,0.02);
	\node at (1, -0.07) {\scriptsize $b$};
	\draw[thick] (3/5,3/4)--(4/5,1/2);
	\end{tikzpicture}
	\begin{tikzpicture}[scale=4]
	\draw (0,0)--(1,0)--(1,1)--(0,1)--(0,0);
	\draw (0,0)--(0.2,0.85)--(0.4,0.5)--(0.6,0.85)--(0.8,0.25)--(1,1);
	\draw (0,-0.02)--(0,0.02);
	\node at (0.5, 0.8) { $g$};
	\node at (0, -0.07) {\scriptsize $a$};
	\draw (0.2,-0.02)--(0.2,0.02);
	\node at (0.2, -0.07) {\scriptsize $c_1$};
	\draw (0.4,-0.02)--(0.4,0.02);
	\node at (0.4, -0.07) {\scriptsize $c_2$};
	\draw (3/5,-0.02)--(3/5,0.02);
	\node at (3/5, -0.07) {\scriptsize $c_3$};
	\draw (4/5,-0.02)--(4/5,0.02);
	\node at (4/5, -0.07) {\scriptsize $c_4$};
	\draw (1,-0.02)--(1,0.02);
	\node at (1, -0.07) {\scriptsize $b$};
	\draw[thick] (0.6,0.85)--(0.8,0.25);
	\draw[thick] (0.2,0.85)--(0.4,0.5);
	\end{tikzpicture}
	\caption{Point $y\in I$ is in a zigzag of $f$ if and only if $y\in (c_3,c_4)$. Point $y\in I$ is in a zigzag of $g$ if and only if $y\in (c_1,c_2)\cup (c_3,c_4)$ (in boldface).}
	\label{fig:zigzag}
\end{figure}
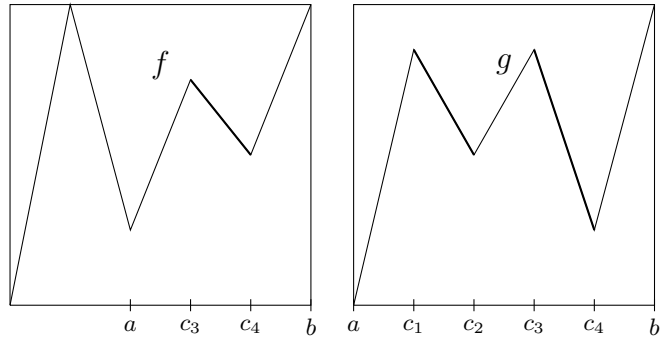

Note that if $\{c_1,\ldots,c_m\}$ are critical points of $f$, then \eg points in $[0,c_1]$ and $[c_m,1]$ are never in a zigzag of $f$. We can say more:

\begin{remark}\label{rem:zigzag}
	Note that if $f(c_k)\in\{0,1\}$, then for every $a<b$ such that $a<c_k<c_{k+1}<b$, $c_k$ is a local minimum or a local maximum of $f|_{[a,b]}$. Thus no point in $[c_k,c_{k+1}]$ is contained in a zigzag of $f$.
\end{remark}

The following lemma gives another criterion which determines when a point is not in a zigzag. It will be used in the proof of Theorem~\ref{thm:main}.

\begin{lemma}\label{lem:speczz}
	Let $f\colon I\to I$ be a piecewise monotone map, and $y\in I$. Assume that there exist $a<b\in I$ such that $y\in[a,b]$, $f(t)\not\in\{f(a),f(b)\}$ for all $t\in (a,b)$, and either
	\begin{enumerate}
		\item $f(a)\in\{0,1\}$ and $f|_{[y,b]}$ is one-to-one, or
		\item $f(b)\in\{0,1\}$, and $f|_{[a,y]}$ is one-to-one.
	\end{enumerate} 
	Then $y$ is not in a zigzag of $f$.
\end{lemma}
\begin{proof}
	Assume first that $(1)$ holds, and $f(a)=0$. Then $f((a,b))=(0,f(b))$, and since $f|_{[y,b]}$ is one-to-one, also $f((y,b))=(f(y),f(b))$. Thus for adjacent critical points $c_k,c_{k+1}$ of $f$ such that $[y,b]\subseteq [c_k,c_{k+1}]$, it holds that $f(c_k)<f(c_{k+1})$. If $y$ is in a zigzag of $f$, there exist $\alpha,\beta\in I$ such that $\alpha<c_k<c_{k+1}<\beta$, and $f|_{[\alpha,\beta]}$ assumes its global maximum in $\alpha$, and global minimum in $\beta$. However, since $f(t)<f(b)$ for all $t\in [a,b)$, it must hold that $\alpha<a$. But then $a\in[\alpha,\beta]$, and since $f(a)=0$, it follows that $a$ is a local minimum of $f|_{[\alpha,\beta]}$, which is a contradiction with $\beta$ being the global minimum.
	
	Assume that $(1)$ holds, with $f(a)=1$. Then $f((a,b))=(f(b),1)$, thus also $f((y,b))=(f(b),f(y))$. It follows that if $[y,b]\subseteq [c_k,c_{k+1}]$, then $f(c_k)>f(c_{k+1})$. So if $y$ is in a zigzag of $f$, there exist $\alpha,\beta\in I$ such that $\alpha<c_k<c_{k+1}<\beta$, $\alpha$ is the global minimum, and $\beta$ is the global maximum of $f|_{[\alpha,\beta]}$. However, since $f(t)>f(b)$ for all $t\in[a,b)$, it must hold that $\alpha<a$. But then $a\in [\alpha,\beta]$, and since $f(a)=1$, it follows that $a$ is a local maximum of $f|_{[\alpha,\beta]}$, which is a contradiction.
	
	Case $(2)$ follows similarly.
\end{proof}

\begin{proposition}\label{prop:comp}
	Let $f,g\colon I\to I$ be piecewise monotone maps. If $y\in I$ is in a zigzag of $g\circ f$, then $y$ is in a zigzag of $f$, or $f(y)$ is in a zigzag of $g$.
\end{proposition}
\begin{proof}
	Assume that $y$ is not in a zigzag of $f$ and $f(y)$ is not in a zigzag of $g$.

	{\bf Claim 1.} We claim that for every  $a<y<b$ such that $g\circ f((a,b))=(g(f(a)),g(f(b)))$, there exists a maximal interval $J\subset [a,b]$ such that $y\in J$, and $g\circ f|_J$ is monotone increasing. 
	
	{\bf (a)} Assume that $f|_{[a,b]}$ is increasing. Thus $f((a,b))=(f(a),f(b))$, and $g|_{[f(a),f(b)]}$ is also increasing, and $g((f(a),f(b)))=(g(f(a)),g(f(b)))$. 
		
	Let $\alpha<\beta$ be such that $a\leq \alpha\leq y\leq\beta\leq b$,  $f|_{[\alpha,\beta]}$ is monotone, and such that $[\alpha,\beta]\ni y$ is a maximal such interval. If $y$ is a critical point of $f$, then we choose $[\alpha,\beta]$ such that $f|_{[\alpha,\beta]}$ is increasing.
		
	Assume that $y$ is not a critical point of $f$.
	If $\alpha=a$, or $\beta=b$, then $f|_{[\alpha,\beta]}$ is obviously increasing. Furthermore, since we assumed that $y$ is not in a zigzag of $f$, if $a<\alpha<\beta<b$, then $f|_{[\alpha,\beta]}$ must also be increasing.
		
	Similarly, let $\gamma<\delta$ be such that $f(a)\leq \gamma\leq f(y)\leq \delta\leq f(b)$, $g|_{[\gamma,\delta]}$ is monotone, and $[\gamma,\delta]\ni f(y)$ is a maximal such interval. Again, if $f(y)$ is a critical point of $g$, we choose $[\gamma,\delta]$ such that $g|_{[\gamma,\delta]}$ is increasing. If $f(y)$ is not a critical point of $g$, since $f(y)$ is not in a zigzag of $g$, similarly as before we conclude that $g|_{[\gamma,\delta]}$ is increasing.
		
	Define $J:=f^{-1}([\gamma,\delta])\cap[\alpha,\beta]$. Then $J$ is an interval, and it is a maximal interval which contains $y$ such that $g\circ f|_J$ is monotone. Moreover, since $J\subset [\alpha,\beta]$, and $f(J)\subset [\gamma,\delta]$, and $f|_{[\alpha,\beta]}$, $g|_{[\gamma,\delta]}$ are monotone increasing, it follows that $g\circ f|_J$ is also monotone increasing.

	{\bf (b)} Assume that $f|_{[a,b]}$ is decreasing. Thus $f((a,b))=(f(b),f(a))$,  $g|_{[f(a),f(b)]}$ is also decreasing, and $g((f(b),f(a)))=(g(f(a)),g(f(b)))$. We take $\alpha,\beta$ as in the previous paragraph, but this time if $y$ is a critical point of $f$, we choose such that $f|_{[\alpha,\beta]}$ is decreasing. Again we conclude that $f|_{[\alpha,\beta]}$ is decreasing in any case. Similarly we find $\gamma,\delta$ as before, but if $f(y)$ is a critical point of $g$, we choose $g|_{[\gamma,\delta]}$ to be decreasing. We conclude that $g|_{[\gamma,\delta]}$ is decreasing in any case. So  $J:=f^{-1}([\gamma,\delta])\cap[\alpha,\beta]$ is again a maximal interval which contains $y$ such that $f|_J$ is monotone. Since in this case $f|_{[\alpha,\beta]}$ and $g|_{[\gamma,\delta]}$ are both decreasing, $g\circ f|_J$ is again monotone increasing.

	{\bf Claim 2.} We claim that for every  $a<y<b$ such that $g\circ f((a,b))=(g(f(b)),g(f(a)))$, there exists a maximal interval $J\subset [a,b]$ such that $y\in J$, and $g\circ f|_J$ is monotone decreasing.
	
	{\bf (a)} Assume that $f|_{[a,b]}$ is increasing. Thus $f((a,b))=(f(a),f(b))$, and $g|_{[f(a),f(b)]}$ is decreasing, $g((f(a),f(b)))=(g(f(b)),g(f(a)))$.  We define $\alpha,\beta,\gamma,\delta$ as before, and this time we conclude that $f|_{[\alpha,\beta]}$ is increasing, and $g|_{[\gamma,\delta]}$ is decreasing. Then  $J:=f^{-1}([\gamma,\delta])\cap[\alpha,\beta]$ is again a maximal interval which contains $y$ such that $f|_J$ is monotone. Since in this case $f|_{[\alpha,\beta]}$ is increasing, and $g|_{[\gamma,\delta]}$ is decreasing, $g\circ f|_J$ is monotone decreasing.
	
	{\bf (b)} Assume that $f|_{[a,b]}$ is decreasing. Thus $f((a,b))=(f(b),f(a))$,  $g|_{[f(a),f(b)]}$ is increasing, and $g((f(b),f(a)))=(g(f(b)),g(f(a)))$. Now $f|_{[\alpha,\beta]}$ is decreasing, and $g|_{[\gamma,\delta]}$ is increasing. Then $J:=f^{-1}([\gamma,\delta])\cap[\alpha,\beta]$ is a maximal interval which contains $y$ such that $f|_J$ is monotone, and $f|_J$ is monotone decreasing.
	
	Recall that if $y$ is in a zigzag of $g\circ f$, then there are $a<y<b$ such that for every maximal interval $J\ni y$ such that $g\circ f|_J$ is monotone, either
	\begin{enumerate}
		\item $g\circ f((a,b))=(g(f(a)),g(f(b)))$, and $g\circ f|_J$ is decreasing, or
		\item $g\circ f((a,b))=(g(f(b)),g(f(a)))$, and $g\circ f|_J$ is increasing.
	\end{enumerate}
	Thus Claims~1 and 2 imply that $y$ is not in a zigzag of $g\circ f$.
\end{proof}

Let $X$ be a continuum and $\nu\colon X\to \nu(X)\subset\R^2$ be an embedding of $X$ in the plane. We say that $\nu$ is a {\em thin embedding} (also called $C$-embedding in \cite{ABC-q}) if for every $\eps>0$ there is a chain $\chain=\{\ell_1,\ldots,\ell_n\}$ of $\nu(X)$ which covers $\nu(X)$, and such that $\ell_i$ is a connected set in $\R^2$ for every $i\in\{1,\ldots,n\}$.

The following theorem gives a connection between accessibility and zigzags in bonding maps.

\begin{theorem}\cite[Theorem~7.3]{ABC-XS}\label{thm:zigzag}
	Let $X=\underleftarrow{\lim}\{I, f_i\}$ where each $f_i\colon I\to I$ 
	is a continuous piecewise monotone surjection. 
	If $x=(x_0, x_1, x_2, \dots)\in X$ is such that for each $i\in\N$, $x_i$ is not inside a zigzag of $f_i$, then there exists a thin embedding $\nu\colon X\to \nu(X)$ of $X$	in the plane such that $\nu(x)$ is an accessible point of $\nu(X)$.
\end{theorem}

\begin{theorem}\label{thm:representations}
	Let $X=\underleftarrow{\lim}(I,f)$, and assume that there exist sequences $(n_i)_{i\in\N}\subset\N$, $(s_i)_{i\in\N},(t_i)_{i\in\N}$, where $s_i,t_i\colon I\to I$ are onto maps for every $i\in\N$, such that $t_i\circ s_i=f^{n_i}$ for every $i\in\N$. Define $g_i=s_i\circ t_{i+1}$ for $i\in\N$; then the diagram in Figure~\ref{fig:comm} commutes. Let $x=(x_0,x_1,x_2,\ldots)\in X$. If $s_i(x_{n_i})$ is not in a zigzag of $g_{i-1}$ for every $i\geq 2$, then there exists a thin embedding $\nu\colon X\to\R^2$ such that $\nu(x)$ is accessible.
\end{theorem}
\begin{proof}
	Let $h\colon X\to\underleftarrow{\lim}(I,g_i)$ be a homeomorphism given by $$h((\xi_0,\xi_1,\xi_2,\ldots))=(s_1(\xi_{n_1}),s_2(\xi_{n_2}),s_3(\xi_{n_3}),\ldots).$$
	By Theorem~\ref{thm:zigzag}, there is a thin embedding $\mu\colon\underleftarrow{\lim}(I,g_i)\to\R^2$ such that $\mu(h(x))$ is accessible. Then $\nu:=\mu\circ h\colon X\to\R^2$ is a thin embedding of $X$, and $\nu(x)$ is accessible.
\end{proof}

\section{Embeddings of Minc's continuum}\label{sec:Minc}

In this section we show how to embed every point of Minc's continuum $X_M$ accessibly. It is important to understand this example since the procedure generalizes to a much wider class of chainable continua. The generalization will be given in the next section.

Recall that $X_M=\underleftarrow{\lim}(I,f_M)$, where $f_M$ is given in Figure~\ref{fig:Minc}. We first construct maps $s,t,s',t'\colon I\to I$ such that $t\circ s=f_M^2$ and $t'\circ s'=f_M^2$. 

We define $s, t\colon I\to I$ as
\vspace{-20pt}
\begin{multicols}{2}
	\begin{equation*}
	s(y):=\begin{cases}
	\frac{7}{18}(1-f^2_M(x)), & y\in[0,\frac{7}{18}]\\
	y, & y\in[\frac{7}{18},1],
	\end{cases}
	\end{equation*}\break
	\begin{equation*}
	t(y):=\begin{cases}
	1-\frac{18}{7}y, & y\in[0,\frac{7}{18}],\\
	f^2_M(y), & y\in[\frac{7}{18},1],
	\end{cases}
	\end{equation*}
\end{multicols}

see Figure~\ref{fig:sandt}. Note that $f_M^2(7/18)=0$, so $s$ and $t$ are well-defined and continuous. Furthermore, if $y\in[0,7/18]$, then $s(y)=7/18(1-f_M^2(y))$, and $s(y)\in [0,7/18]$. So $t(s(y))=1-\frac{18}{7}(\frac{7}{18}(1-f_M^2(y)))=f_M^2(y)$. If $y\in[7/18,1]$, then $s(y)=y$, hence $t(s(y))=t(y)=f_M^2(y)$. It follows that $t\circ s=f_M^2$.

\begin{figure}[h!]
	\centering
	\begin{tikzpicture}[xscale=5,yscale=5]
	\draw (0,0)--(0,1)--(1,1)--(1,0)--(0,0);
	\draw[dashed] (7/18,0)--(7/18,1);
	\draw[dashed] (0,7/18)--(1,7/18);
	\node[below] at (7/18,0) {\small $\frac{7}{18}$};
	\node[left] at (0,7/18) {\small $\frac{7}{18}$};
	\draw[dashed] (1/9,0)--(1/9,1);
	\draw[dashed] (0,7/54)--(1,7/54);
	\node[below] at (1/9,0) {\small $\frac{1}{9}$};
	\node[left] at (0,7/54) {\small $\frac{7}{54}$};
	\draw[dashed] (2/9,0)--(2/9,1);
	\draw[dashed] (0,14/54)--(1,14/54);
	\node[below] at (2/9,0) {\small $\frac{2}{9}$};
	\node[left] at (0,14/54) {\small $\frac{14}{54}$};
	\draw[dashed] (1/3,0)--(1/3,1);
	\node[below] at (1/3,0) {\small $\frac{1}{3}$};
	
	\node[below] at (7/18,-0.175) {};
	\node[above] at (11/18,1.1) {};

	\draw ( 0 , 0.3888888888888889 ) --
	( 0.1111111111111111 , 0 ) --
	( 0.14814814814814814 , 0.25925925925925924 ) --
	( 0.18518518518518517 , 0.12962962962962962 ) --
	( 0.2222222222222222 , 0.3888888888888889 ) --
	( 0.3333333333333333 , 0 ) --
	( 0.3888888888888889 , 0.3888888888888889 ) --
	( 1 , 1 );
	\end{tikzpicture}
	\begin{tikzpicture}[xscale=5,yscale=5]
	\draw (0,0)--(0,1)--(1,1)--(1,0)--(0,0);
	\draw[dashed] (1/3,0)--(1/3,1);
	\draw[dashed] (2/3,0)--(2/3,1);
	\draw[dashed] (0,1/3)--(1,1/3);
	\draw[dashed] (0,2/3)--(1,2/3);
	\node[below] at (1/3,0) {\small $\frac 13$};
	\node[below] at (2/3,0) {\small $\frac 23$};
	\node[left] at (0,1/3) {\small $\frac 13$};
	\node[left] at (0,2/3) {\small $\frac 23$};
	\draw[dashed,thin] (4/9,0)--(4/9,1);
	\node[below] at (4/9,0) {\small $\frac 49$};
	\draw[dashed,thin] (5/9,0)--(5/9,1);
	\node[below] at (5/9,0) {\small $\frac 59$};

	\node[below] at (7/18,-0.1) {\small $\frac{7}{18}$};
	\draw[->] (7/18,-0.11)--(7/18,-0.01);
	
	\draw ( 0 , 1 ) --
	( 0.3888888888888889 , 6.661338147750939e-16 ) --
	( 0.4074074074074074 , 0.6666666666666666 ) --
	( 0.42592592592592593 , 0.3333333333333346 ) --
	( 0.4444444444444444 , 0.9999999999999998 ) --
	( 0.48148148148148145 , 0.3333333333333336 ) --
	( 0.5185185185185185 , 0.666666666666666 ) --
	( 0.5555555555555556 , 0.0 ) --
	( 0.5740740740740741 , 0.6666666666666666 ) --
	( 0.5925925925925926 , 0.333333333333334 ) --
	( 0.6111111111111112 , 0.9999999999999989 ) --
	( 0.6666666666666666 , 0.0 ) --
	( 0.7777777777777778 , 0.9999999999999993 ) --
	( 0.8148148148148148 , 0.3333333333333336 ) --
	( 0.8518518518518519 , 0.6666666666666666 ) --
	( 0.8888888888888888 , 6.661338147750939e-16 ) --
	( 1 , 1.0 );
	\end{tikzpicture}
	\vspace{-15pt}
	\caption{Graphs of maps $s$ and $t$. Note that $t\circ s=f_M^2$.}
	\label{fig:sandt}
\end{figure}

We define maps $s',t'\colon I\to I$ as follows:
\vspace{-20pt}
\begin{multicols}{2}
	\begin{equation*}
	s'(y):=\begin{cases}
		y, & y\in[0,\frac{11}{18}]\\
		1-\frac{7}{18}f_M^2(y), & y\in[\frac{11}{18},1],
		\end{cases}
	\end{equation*}\break
	\begin{equation*}
	t'(y):=\begin{cases}
		f^2_M(y), & y\in[0,\frac{11}{18}],\\
		\frac{18}{7}(1-y), & y\in[\frac{11}{18},1],
		\end{cases}
	\end{equation*}
\end{multicols}
see Figure~\ref{fig:s'andt'}. Since $f_M^2(11/18)=1$, $s'$ and $t'$ are well-defined and continuous. Note also that for $y\in[0,11/18]$ we have $t'(s'(y))=t'(y)=f_M^2(y)$, and for $y\in[11/18,1]$, also $s'(y)\in[11/18,1]$, and thus $t'(s'(y))=t'(1-7/18(f_M^2(y)))=f_M^2(y)$. It follows that $t'\circ s'=f_M^2$.

\begin{figure}[h!]
	\centering
	\begin{tikzpicture}[xscale=5,yscale=5]
	\draw (0,0)--(0,1)--(1,1)--(1,0)--(0,0);
	\draw[dashed] (11/18,0)--(11/18,1);
	\draw[dashed] (0,11/18)--(1,11/18);
	\node[below] at (11/18,0) {\small $\frac{11}{18}$};
	\node[left] at (0,11/18) {\small $\frac{11}{18}$};
	\draw[dashed] (7/9,0)--(7/9,1);
	\draw[dashed] (0,47/54)--(1,47/54);
	\node[below] at (7/9,0) {\small $\frac{7}{9}$};
	\node[left] at (0,47/54) {\small $\frac{47}{54}$};
	\draw[dashed] (8/9,0)--(8/9,1);
	\draw[dashed] (0,40/54)--(1,40/54);
	\node[below] at (8/9,0) {\small $\frac{8}{9}$};
	\node[left] at (0,40/54) {\small $\frac{40}{54}$};
	\draw[dashed] (2/3,0)--(2/3,1);
	\node[below] at (2/3,0) {\small $\frac{2}{3}$};
	
	\node[below] at (7/18,-0.175) {};
	\node[above] at (11/18,1.1) {};

	\draw ( 0 , 0 ) --
	( 0.6111111111111112 , 0.6111111111111112 ) --
	( 0.6666666666666666 , 1 ) --
	( 0.7777777777777778 , 0.6111111111111112 ) --
	( 0.8148148148148148 , 0.8703703703703703 ) --
	( 0.8518518518518519 , 0.7407407407407407 ) --
	( 0.8888888888888888 , 1 ) --
	( 1 , 0.6111111111111112 );
	\end{tikzpicture}
	\begin{tikzpicture}[xscale=5,yscale=5]
	\draw (0,0)--(0,1)--(1,1)--(1,0)--(0,0);
	\draw[dashed] (1/3,0)--(1/3,1);
	\draw[dashed] (2/3,0)--(2/3,1);
	\draw[dashed] (0,1/3)--(1,1/3);
	\draw[dashed] (0,2/3)--(1,2/3);
	\node[below] at (1/3,0) {\small $\frac 13$};
	\node[below] at (2/3,0) {\small $\frac 23$};
	\node[left] at (0,1/3) {\small $\frac 13$};
	\node[left] at (0,2/3) {\small $\frac 23$};
	\draw[dashed,thin] (4/9,0)--(4/9,1);
	\node[below] at (4/9,0) {\small $\frac 49$};
	\draw[dashed,thin] (5/9,0)--(5/9,1);
	\node[below] at (5/9,0) {\small $\frac 59$};
	
	\draw[dashed] (11/18,0)--(11/18,1);

	\node[below] at (7/18,-0.175) {\small $ $};
	\node[above] at (11/18,1.1) {\small $\frac{11}{18}$};
	\draw[->] (11/18,1.11)--(11/18,1.01);
	
	\draw ( 0 , 0.0 ) --
	( 0.1111111111111111 , 1.0 ) --
	( 0.14814814814814814 , 0.33333333333333326 ) --
	( 0.18518518518518517 , 0.6666666666666666 ) --
	( 0.2222222222222222 , 0.0 ) --
	( 0.3333333333333333 , 1.0 ) --
	( 0.3888888888888889 , 6.661338147750939e-16 ) --
	( 0.4074074074074074 , 0.6666666666666666 ) --
	( 0.42592592592592593 , 0.3333333333333346 ) --
	( 0.4444444444444444 , 0.9999999999999998 ) --
	( 0.48148148148148145 , 0.3333333333333336 ) --
	( 0.5185185185185185 , 0.666666666666666 ) --
	( 0.5555555555555556 , 0.0 ) --
	( 0.5740740740740741 , 0.6666666666666666 ) --
	( 0.5925925925925926 , 0.333333333333334 ) --
	( 0.6111111111111112 , 0.9999999999999989 ) --
	( 1 , 0 );
	\end{tikzpicture}
	\vspace{-20pt}
	\caption{Graphs of maps $s'$ and $t'$. Note that $t'\circ s'=f_M^2$.}
	\label{fig:s'andt'}
\end{figure}

Now let $x=(x_0,x_1,x_2,x_3,\ldots)\in X_M$. We will construct a planar embedding $\nu_x\colon X\to\R^2$ such that $\nu_x(x)$ is accessible.

Note that $(x_0,x_2,x_4,\ldots)\in\underleftarrow{\lim}(I,f_M^2)$. For every $i\in\N$ we define maps $s_i,t_i\colon I\to I$ as one of $s,s',t,t'$, depending on the position of the $(2i)$th coordinate of $x$ as follows: 
\begin{multicols}{2}
	\begin{equation*}
	s_i=\begin{cases}
	s', & x_{2i}\in[0,\frac{7}{18}],\\
	s, & x_{2i}\in (\frac{7}{18},1].
	\end{cases}
	\end{equation*}\break
	\begin{equation*}
	t_i=\begin{cases}
	t', & x_{2i}\in[0,\frac{7}{18}],\\
	t, & x_{2i}\in (\frac{7}{18},1].
	\end{cases}
	\end{equation*}
\end{multicols}

Note that $t_i\circ s_i=f_M^2$, $s_i(x_{2i})=x_{2i}$, and $x_{2i}$ is not in a zigzag of $s_i$ for every $i\in\N$. Furthermore, for $i\in\N$ we define $g_i:=s_i\circ t_{i+1}$, see the commutative diagram in Figure~\ref{fig:comm2}.

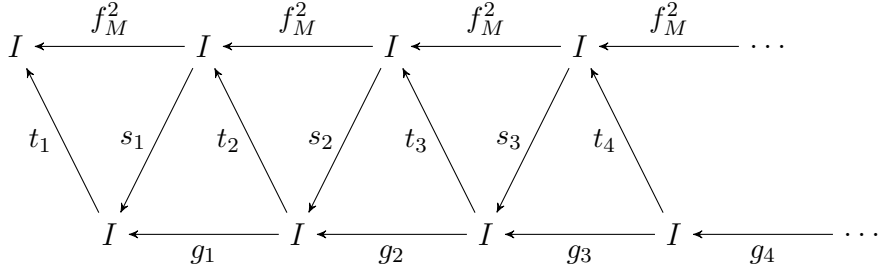
\begin{figure}[!ht]
	\centering
	\begin{tikzpicture}[->,>=stealth',auto, scale=2.5]
	\node (1) at (0,1) {$I$};
	\node (2) at (0.5,0) {$I$};
	\node (3) at (1,1) {$I$};
	\node (4) at (1.5,0) {$I$};
	\node (5) at (2,1) {$I$};
	\node (6) at (2.5,0) {$I$};
	\node (7) at (3,1) {$I$};
	\node (8) at (3.5,0) {$I$};
	\node (9) at (4,1) {$\ldots$};
	\node (10) at (4.5,0) {$\ldots$};
	\draw [->] (3) -- node[above]{\small $f_M^2$} (1);
	\draw [->] (5) -- node[above]{\small $f_M^2$} (3);
	\draw [->] (7) -- node[above]{\small $f_M^2$} (5);
	\draw [->] (9) -- node[above]{\small $f_M^2$} (7);
	\draw [->] (4) -- node[below]{\small $g_1$} (2);
	\draw [->] (6) -- node[below]{\small $g_2$} (4);
	\draw [->] (8) -- node[below]{\small $g_3$} (6);
	\draw [->] (10) -- node[below]{\small $g_4$} (8);
	\draw [->] (2) -- node[left]{\small $t_1$} (1);
	\draw [->] (3) -- node[left]{\small $s_1$} (2);
	\draw [->] (4) -- node[left]{\small $t_2$} (3);
	\draw [->] (5) -- node[left]{\small $s_2$} (4);
	\draw [->] (6) -- node[left]{\small $t_3$} (5);
	\draw [->] (7) -- node[left]{\small $s_3$} (6);
	\draw [->] (8) -- node[left]{\small $t_4$} (7);
	\end{tikzpicture}
	\caption{Commutative diagram. Minc's continuum $X_M$ is homeomorphic to $\protect \underleftarrow{\lim}(I,g_i)$, and coordinates of $x$ will not be in zigzags of bonding maps $g_i$.}
	\label{fig:comm2}
\end{figure}

\begin{lemma}\label{lem:mincz}
	If $y$ is in a zigzag of $g_i=s_i\circ t_{i+1}$, then $t_{i+1}(y)$ is in a zigzag of $s_i$.
\end{lemma}
\begin{proof}
	By Proposition~\ref{prop:comp}, $y$ is in a zigzag of $t_{i+1}$, or $t_{i+1}(y)$ is in a zigzag of $s_i$. Assume that $y$ is in a zigzag of $t_{i+1}$. Then Remark~\ref{rem:zigzag} implies that there are critical points $c_k,c_{k+1}$ of $t_{i+1}$ such that $c_k<y<c_{k+1}$, $t_{i+1}|_{[c_k,c_{k+1}]}$ is monotone, and $t_{i+1}([c_k,c_{k+1}])=[1/3,2/3]$. Furthermore, $s_i|_{[1/3,2/3]}$ is monotone, and $s_i([1/3,2/3])=[0,2/3]$ if $s_i=s$, or $s_i([1/3,2/3])=[1/3,1]$ if $s_i=s'$. In any case, $g_i|_{[c_k,c_{k+1}]}=s_i\circ t_{i+1}|_{[c_k,c_{k+1}]}$ is monotone, and at least one of $g_i(c_k),g_i(c_{k+1})$ is in $\{0,1\}$. By Remark~\ref{rem:zigzag}, $y$ is not in a zigzag of $g_i$, which is a contradiction. It follows that $t_{i+1}(y)$ must be in a zigzag of $s_i$.
\end{proof}

\begin{lemma}\label{lem:zznozz}
	For every $i\geq 0$ it holds that $g_i(x_{2(i+1)})=x_{2i}$ and $x_{2(i+1)}$ not in a zigzag of $g_i$.
\end{lemma}
\begin{proof}
	We have $g_i(x_{2(i+1)})=g_i(s_{i+1}(x_{2(i+1)}))=s_i\circ t_{i+1}\circ s_{i+1}(x_{2(i+1)})=s_i\circ f_M^2(x_{2(i+1)})=s_i(x_{2i})=x_{2i}$. Furthermore, assume that $x_{2(i+1)}$ is in a zigzag of $g_i$. By Lemma~\ref{lem:mincz}, $t_{i+1}(x_{2(i+1)})=x_{2i}$ is in a zigzag of $s_i$. That is a contradiction.
\end{proof}

\begin{theorem}\label{thm:minc}
	For every $x\in X_M$ there exists a thin planar embedding $\nu_{x}\colon X_M\to \R^2$ such that $\nu_{x}(x)$ is accessible.
\end{theorem}
\begin{proof}
	Proof follows directly from Theorem~\ref{thm:representations} and Lemma~\ref{lem:zznozz}.
\end{proof}

\begin{example}
	In particular, let $x=(1/2,1/2,1/2,\ldots)$. Then $s_i=s$, $t_i=t$, and $g_i=:g=s\circ t$ for every $i\in\N$. Thus there is a homeomorphism $h\colon X_M\to\underleftarrow{\lim}(I,g)$ such that $h(x)=(s(1/2),s(1/2),\ldots)=(1/2,1/2,\ldots)\in\underleftarrow{\lim}(I,g)$, and $1/2$ is not in a zigzag of $g$, see the graph of $g$ in Figure~\ref{fig:g}. We can then easily embed $\underleftarrow{\lim}(I,g)$ in the plane with $(1/2,1/2,\ldots)$ accessible.  
\end{example}

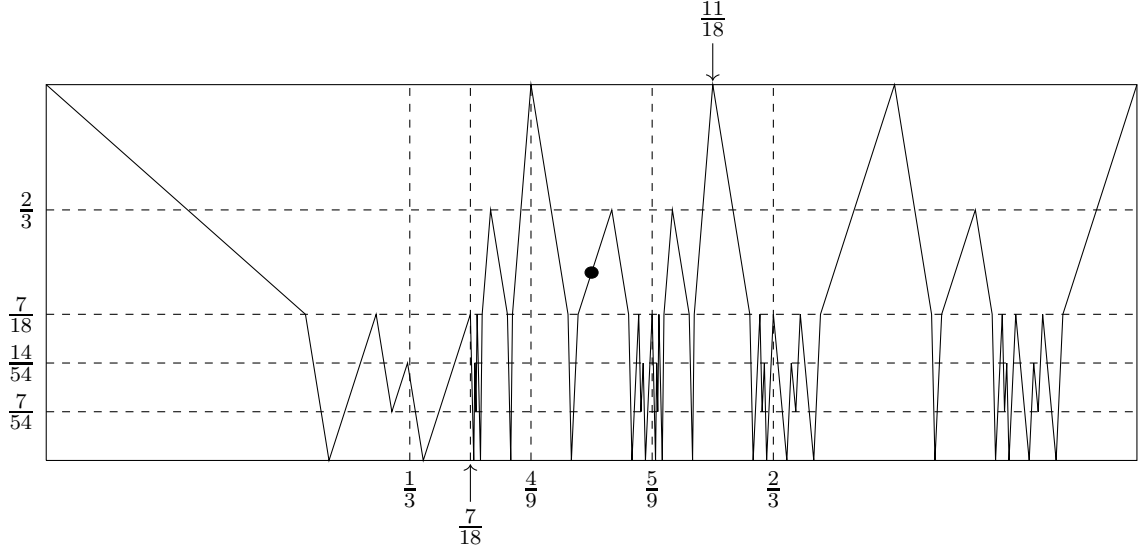
\begin{figure}
	\centering
	\hspace{-15pt}
	\begin{tikzpicture}[xscale=14.5,yscale=5]
	\draw (0,0)--(0,1)--(1,1)--(1,0)--(0,0);
	\draw[dashed] (1/3,0)--(1/3,1);
	\draw[dashed] (2/3,0)--(2/3,1);
	\draw[dashed] (0,7/18)--(1,7/18);
	\draw[dashed] (0,2/3)--(1,2/3);
	\node[below] at (1/3,0) {\small $\frac 13$};
	\node[below] at (2/3,0) {\small $\frac 23$};
	\node[left] at (0,7/18) {\small $\frac{7}{18}$};
	\node[left] at (0,2/3) {\small $\frac 23$};
	\draw[dashed,thin] (4/9,0)--(4/9,1);
	\node[below] at (4/9,0) {\small $\frac 49$};
	\draw[dashed,thin] (5/9,0)--(5/9,1);
	\node[below] at (5/9,0) {\small $\frac 59$};
	\draw[dashed] (7/18,0)--(7/18,1);
	
	\draw[fill,xscale=0.4] (1/2+0.75,1/2) circle (0.015);
	
	\node[below] at (7/18,-0.1) {\small $\frac{7}{18}$};
	\draw[->] (7/18,-0.11)--(7/18,-0.01);
	\node[above] at (11/18,1.1) {\small $\frac{11}{18}$};
	\draw[->] (11/18,1.11)--(11/18,1.01);
	
	\draw[dashed] (0,7/54)--(1,7/54);
	\node[left] at (0,7/54) {\small $\frac{7}{54}$};
	\draw[dashed] (0,14/54)--(1,14/54);
	\node[left] at (0,14/54) {\small $\frac{14}{54}$};
	
	\draw ( 0 , 1.0 ) --
	( 0.23765432098765452 , 0.38888888888888773 ) --
	( 0.25925925925925947 , 1.6653345369377348e-16 ) --
	( 0.30246913580246937 , 0.3888888888888882 ) --
	( 0.3168724279835394 , 0.1296296296296302 ) --
	( 0.3312757201646093 , 0.25925925925925847 ) --
	( 0.34567901234567927 , 1.6653345369377348e-16 ) --
	( 0.3888888888888889 , 0.38888888888888695 ) --
	( 0.3919753086419753 , 2.6645352591003757e-15 ) --
	( 0.39300411522633744 , 0.25925925925925897 ) --
	( 0.3940329218106996 , 0.12962962962963623 ) --
	( 0.3950617283950617 , 0.3888888888888876 ) --
	( 0.39814814814814814 , 2.7200464103316323e-15 ) --
	( 0.39969135802469136 , 0.3888888888888895 ) --
	( 0.4074074074074074 , 0.6666666666666667 ) --
	( 0.42283950617283955 , 0.3888888888888891 ) --
	( 0.42592592592592593 , 8.937295348232507e-15 ) --
	( 0.4274691358024691 , 0.388888888888885 ) --
	( 0.4444444444444444 , 0.9999999999999998 ) --
	( 0.4783950617283951 , 0.3888888888888846 ) --
	( 0.48148148148148145 , 1.942890293094023e-15 ) --
	( 0.48765432098765427 , 0.38888888888888773 ) --
	( 0.5185185185185185 , 0.666666666666666 ) --
	( 0.5339506172839505 , 0.38888888888888967 ) --
	( 0.537037037037037 , 4.662936703425656e-15 ) --
	( 0.5432098765432098 , 0.3888888888888871 ) --
	( 0.5452674897119342 , 0.1296296296296306 ) --
	( 0.5473251028806584 , 0.25925925925925575 ) --
	( 0.5493827160493827 , 1.942890293094024e-15 ) --
	( 0.5555555555555556 , 0.3888888888888889 ) --
	( 0.558641975308642 , 4.760081218080359e-15 ) --
	( 0.5596707818930041 , 0.25925925925925614 ) --
	( 0.5606995884773662 , 0.12962962962963498 ) --
	( 0.5617283950617284 , 0.3888888888888841 ) --
	( 0.5648148148148149 , 1.3988810110276968e-14 ) --
	( 0.566358024691358 , 0.38888888888887724 ) --
	( 0.5740740740740741 , 0.6666666666666666 ) --
	( 0.5895061728395061 , 0.3888888888888898 ) --
	( 0.5925925925925926 , 5.051514762044461e-15 ) --
	( 0.5941358024691358 , 0.3888888888888894 ) --
	( 0.6111111111111112 , 0.9999999999999991 ) --
	( 0.6450617283950617 , 0.3888888888888885 ) --
	( 0.6481481481481481 , 9.43689570931383e-16 ) --
	( 0.654320987654321 , 0.38888888888888273 ) --
	( 0.6563786008230452 , 0.12962962962963157 ) --
	( 0.6584362139917695 , 0.25925925925925536 ) --
	( 0.6604938271604938 , 4.274358644806853e-15 ) --
	( 0.6666666666666666 , 0.3888888888888889 ) --
	( 0.6790123456790124 , 2.6229018956769323e-15 ) --
	( 0.6831275720164609 , 0.25925925925925863 ) --
	( 0.6872427983539094 , 0.12962962962962982 ) --
	( 0.691358024691358 , 0.3888888888888872 ) --
	( 0.7037037037037037 , 7.771561172376093e-16 ) --
	( 0.7098765432098766 , 0.3888888888888892 ) --
	( 0.7777777777777778 , 0.9999999999999991 ) --
	( 0.8117283950617283 , 0.3888888888888894 ) --
	( 0.8148148148148148 , 1.942890293094023e-15 ) --
	( 0.8209876543209876 , 0.3888888888888893 ) --
	( 0.8518518518518519 , 0.6666666666666666 ) --
	( 0.8672839506172839 , 0.38888888888888945 ) --
	( 0.8703703703703703 , 1.942890293094023e-15 ) --
	( 0.8765432098765432 , 0.38888888888888695 ) --
	( 0.8786008230452674 , 0.12962962962963584 ) --
	( 0.8806584362139918 , 0.25925925925925924 ) --
	( 0.8827160493827161 , 2.886579864025407e-15 ) --
	( 0.8888888888888888 , 0.38888888888888695 ) --
	( 0.9012345679012345 , 2.914335439641036e-16 ) --
	( 0.905349794238683 , 0.2592592592592581 ) --
	( 0.9094650205761315 , 0.12962962962963098 ) --
	( 0.9135802469135802 , 0.38888888888888745 ) --
	( 0.9259259259259258 , 7.771561172376096e-16 ) --
	( 0.9320987654320987 , 0.3888888888888885 ) --
	( 1 , 0.9999999999999998 );
	\end{tikzpicture}
	\vspace{-10pt}
	\caption{Graph of $g=s\circ t$, with $x$ axis expanded for clarity. Note that $1/2$ is not in a zigzag of $g$.}
	\label{fig:g}
\end{figure}

\section{Post-critically finite locally eventually onto bonding map $f$}\label{sec:general}

We start this section with a theorem which generalizes previous ideas to a much larger class of chainable continua. We will then show that if $f$ is piecewise monotone and post-critically finite locally eventually onto map, then for an arbitrary $x\in\underleftarrow{\lim}(I,f)=:X_f$ there is a planar embedding $\nu_x\colon X_f\to \R^2$ such that $\nu_x(x)$ is an accessible point of $\nu_x(X_f)$.
Recall that $f$ is {\em piecewise monotone} if there is finitely many points $0=c_0<c_1<\ldots<c_m<c_{m+1}=1$ such that $f$ is strictly monotone on $[c_i,c_{i+1}]$ for every $i\in\{0,\ldots,m\}$. The set $C=\{c_0,c_1,\ldots,c_m,c_{m+1}\}$ will be referred to as {\em critical set}.

\begin{definition}
	Let $f\colon I\to I$ be a piecewise monotone map, and $y\in I$ be a non-critical point of $f$. Then there is a maximal interval $J:=J(f,y)\subset I$ such that $y\in J$, and $f|_J$ is one-to-one. If $y$ is a critical point of $f$, then there are two maximal intervals $J_1,J_2\ni y$ such that $f|_{J_1}, f|_{J_2}$ are one-to-one. It holds that $J_1\cap J_2=\{y\}$, and exactly one of $f|_{J_1},f|_{J_2}$ is monotone increasing. We define $J(f,y)=J_1$ if $f(J_1)\subseteq f(J_2)$, or $J(f,y)=J_2$ otherwise. For $y\in I$  we denote $B(f,y):=f(J(f,y))$ and call it an {\em $f$-branch of $y$}. 
\end{definition}

\begin{theorem}\label{thm:main}
	Let $X=\underleftarrow{\lim}(I,f_i)$, where every $f_i$ is piecewise monotone onto map, and let $x=(x_0,x_1,x_2,\ldots)\in X$. Assume that the following conditions are satisfied:
	\begin{enumerate}
		\item there exist $a<b\in I$ such that $B(f_i,x_i)=[a,b]$ for every $i\in\N$,
		\item there is $\eps>0$ such that $[a,a+\eps)\cap\{x_i:i\geq 0\}=\emptyset$, or $\{x_i:i\in\N\}\cap(b-\eps,b]=\emptyset$,
		\item for every interval $J\subset I$ of diameter $\geq\eps/2$ it holds that $f_i(J)=I$ for every $i\in\N$.
	\end{enumerate}
	Then there exists a thin embedding $\nu_x\colon X\to\R^2$ such that $\nu_x(x)$ is an accessible point of $\nu_x(X)$.
\end{theorem}
\begin{proof}
	For every $i\in\N$ we will define $s_i,t_i\colon I\to I$ such that $t_i\circ s_i=f_i$, and such that $s_i(x_i)$ is not in a zigzag of $s_{i-1}\circ t_i$. The reader is encouraged to recall the maps $s,t,s',t'$ from the previous section.
	
	{\bf Case 1.} Assume that $[a,a+\eps)\cap\{x_i: i\geq 0\}=\emptyset$ in $(2)$. By $(3)$, for every $i\in\N$ there is $[\alpha_i,\beta_i]\subset [a,a+\eps)$ such that $f_i(\alpha_i)=1$, $f_i(\beta_i)=0$. We define $s_i,t_i\colon I\to I$ as follows
	\vspace{-20pt}
	\begin{multicols}{2}
		\begin{equation*}
		s_i(y):=\begin{cases}
		\beta_i(1-f_i(y)), & y\in[0,\beta_i],\\
		y, & y\in[\beta_i,1].
		\end{cases}
		\end{equation*}\break
		\begin{equation*}
		t_i(y):=\begin{cases}
		1-\frac{1}{\beta_i}y, & y\in [0,\beta_i],\\
		f_i(y), & y\in[\beta_i,1],
		\end{cases}
		\end{equation*}
	\end{multicols}
	Since $f_i(\beta_i)=0$, it follows that $s_i,t_i$ are well-defined and continuous. Furthermore, if $y\in[0,\beta_i]$, then $\beta_i(1-f_i(y))\in[0,\beta_i]$, so $t_i(s_i(y))=1-\frac{1}{\beta_i}(\beta_i(1-f_i(y)))=f_i(y)$. Thus $t_i\circ s_i(y)=f_i(y)$ for every $x\in I$. Note that $s_i(x_i)=x_i$ and Remark~\ref{rem:zigzag} implies that $x_i$ is not in a zigzag of $f_i$, for every $i\in\N$.
	
	We claim that $x_i=s_i(x_i)$ is not in a zigzag of $s_{i-1}\circ t_i$. Note first that if $x_i$ is not in a zigzag of $f_i$, then it is also not in a zigzag of $t_i$. Since also $t_i(x_i)=x_{i-1}$ is not in a zigzag of $s_{i-1}$, Proposition~\ref{prop:comp} implies that $x_i$ is not in a zigzag of $s_{i-1}\circ t_i$. Assume that $x_i$ is in a zigzag of $f_i$. In particular, Remark~\ref{rem:zigzag} implies that $J(f_i,x_i)$ does not contain $\beta_i$, and thus $J(f_i,x_i)\subset(\beta_i,1]$. So $J(t_i,x_i)=J(f_i,x_i)$, and, by $(1)$, $t_i|_{J(f_i,x_i)}\colon J(f_i,x_i)\to [a,b]$ is one-to-one. Let $\alpha_{i-1}'\in[\alpha_{i-1},\beta_{i-1})$ be the largest such that $s_{i-1}(\alpha_{i-1}')=0$. Since $s_{i-1}(\alpha_{i-1})=0$, such $\alpha_{i-1}'$ exists. Moreover, $s_{i-1}((\alpha_{i-1}',b))=(0,b)$, and $s_{i-1}|_{[x_{i-1},b]}$ is one-to-one. Let $J'\subset J(f_i,x_i)$ be such that $t_i|_{J'}\colon J'\to [\alpha_{i-1}',b]$ is a homeomorphism. Then $x_i\in J'$ and Lemma~\ref{lem:speczz} implies that $x_i$ is not in a zigzag of $s_{i-1}\circ t_i$.
	
	Thus there is a homeomorphism $h\colon \underleftarrow{\lim}(I,f_i)\to\underleftarrow{\lim}(I,s_{i-1}\circ t_i)$ given by $(\xi_0,\xi_1,\xi_2,\ldots)\mapsto(s_1(\xi_1),s_2(\xi_2),\ldots)=(\xi_1,\xi_2,\ldots)$. In particular, $(x_1,x_2,\ldots)\in\underleftarrow{\lim}(I,s_{i-1}\circ t_i)$, and since $x_i$ is not contained in a zigzag of $s_{i-1}\circ t_i$ for every $i\geq 2$, Theorem~\ref{thm:zigzag} implies that there exists a thin planar embedding of $\underleftarrow{\lim}(I,s_{i-1}\circ t_i)$ in which $x$ is accessible. Theorem~\ref{thm:representations} finishes the proof in this case.
	
	{\bf Case 2.} Assume that $[a,a+\eps)\cap\{x_i:i\geq 0\}\neq\emptyset$. By $(2)$, $\{x_i: i\geq 0\}\cap (b-\eps,b]=\emptyset$. By $(3)$, for every $i\in\N$ there is $[\beta_i, \gamma_i]\subset(b-\eps,b]$ such that $f_i(\beta_i)=1$, and $f_i(\gamma_i)=0$. We define maps $s_i,t_i\colon I\to I$ as
	\begin{multicols}{2}
		\begin{equation*}
		s_i(y):=\begin{cases}
		y, & y\in[0,\beta_i],\\
		1-(1-\beta_i)f_i(y), & y\in[\beta_i,1].
		\end{cases}
		\end{equation*}\break
		\begin{equation*}
		t_i(y):=\begin{cases}
		f_i(y), & y\in [0,\beta_i],\\
		\frac{1}{1-\beta_i}(1-y), & y\in[\beta_i,1],
		\end{cases}
		\end{equation*}
	\end{multicols}
	Since $f_i(\beta_i)=1$, it follows that $s_i,t_i$ are continuous. Moreover, if $y\in[\beta_i,1]$, then $s_i(y)=1-(1-\beta_i)f_i(y)\in[\beta_i,1]$, so $t_i(s_i(y))=f_i(y)$. It follows that $t_i\circ s_i=f_i$. Moreover, $s_i(x_i)=x_i$, and Remark~\ref{rem:zigzag} implies that $x_i$ is not in a zigzag of $s_i$ for every $i\in\N$.
	
	We again claim that $x_i=s_i(x_i)$ is not in a zigzag of $s_{i-1}\circ t_i$ for every $i\geq 2$. If $x_i$ is not in a zigzag of $f_i$, then it is also not in a zigzag of $t_i$, so by Proposition~\ref{prop:comp}, it is not in a zigzag of $s_{i-1}\circ t_i$. If $x_i$ is in a zigzag of $f_i$, then Remark~\ref{rem:zigzag} implies that $J(f_i,x_i)\subset [0,\beta_i)$, and again $J(t_i,x_i)=J(f_i,x_i)$, and $t_i|_{J(t_i,x_i)}\colon J(t_i,x_i)\to [a,b]$ is a homeomorphism. Since $s_{i-1}(\gamma_{i-1})=1$, there is the smallest $\gamma'_{i-1}\in[\beta_{i-1},\gamma_{i-1}]$ such that $s_{i-1}(\gamma'_{i-1})=1$. We take $J'\subset J(t_i,x_i)$ such that $t_i|_{J'}\colon J'\to [a,\gamma'_{i-1}]$ is a homeomorphism. Since $x_i\in J'$, $s_{i-1}((a,\gamma'_{i-1}))=(a,1)$, and $s_{i-1}|_{[a,x_{i-1}]}$ is one-to-one, Lemma~\ref{lem:speczz} again implies that $x_i$ is not in a zigzag of $s_{i-1}\circ t_i$. The proof finishes the same as in Case 1.
\end{proof}

\begin{definition}
	An onto map $f\colon I\to I$ is called {\em locally eventually onto (leo)} if for every interval $J\subset I$ there is $n\in\N$ such that $f^n(J)=I$.
\end{definition}

\begin{lemma}
	If $f$ is leo, then for every $\eps>0$ there is $N\in\N$ such that for every interval $J\subset I$ with $\diam(J)\geq\eps$ it holds that $f^n(J)=I$ for every $n\geq N$.
\end{lemma}
\begin{proof}
	We find intervals $\{J_1,\ldots, J_k\}$ where $\diam J_i< \eps/2$, and $\cup_{i=1}^k J_i=I$. Since $f$ is leo, we can find $N\in\N$ such that $f^N(J_i)=I$ for every $i\in\{1,\ldots,k\}$. Then also $f^n(J_i)=I$ for every $n\geq N$ and $i\in\{1,\ldots,k\}$. Then we note that every interval $J\subset I$ such that $\diam(J)\geq\eps$ contains at least one $J_i$, so $f^n(J)=I$ for every $n\geq N$. 
\end{proof}

\begin{definition}
	Let $f$ be a piecewise monotone map with critical set $\{0=c_0<c_1<\ldots<c_m<c_{m+1}=1\}$. We say that $f$ is {\em post-critically finite} if every $c_i$ is eventually periodic, \ie for every $i\in\{0,\ldots,m+1\}$ there are $j(i)\in\N$, and $k(i)\geq 0$ such that $f^{j(i)+k(i)}(c_i)=f^{k(i)}(c_i)$. 
\end{definition} 

\begin{remark}\label{remark:branches}
	Assume that $f$ is piecewise monotone with critical set $\{0=c_0<c_1<\ldots<c_n<c_{m+1}=1\}$, and assume $f$ is post-critically finite. Then note that for every $n\in\N$ and $x\in I$, the endpoints of $B(f^n,x)$ belong to the set $\{f^k(c_i): i\in\{0,\ldots,m+1\}, k\in\N\}$, which is a finite set. Thus there is only finitely many types of branches in all iterates of $f$, \ie the set $\{B(f^n,x): n\in\N, x\in I\}$ is finite. 
\end{remark}

\begin{lemma}\label{lem:postcritfin}
	Let $X=\underleftarrow{\lim}(I,f)$, where $f$ is piecewise monotone leo map which is post-critically finite, and let $x=(x_0,x_1,x_2,\ldots)\in X$. Then there is a strictly increasing sequence $(n_i)_{i\geq 0}\subset\N$, there are $a<b\in I$, and $\eps>0$ such that
	\begin{enumerate}
		\item $B(f^{n_i-n_{i-1}},x_{n_i})=[a,b]$, for all $i\in\N$,
		\item $[a,a+\eps)\cap\{x_{n_i}:i\in\N\}=\emptyset$ or $\{x_{n_i}:i\in\N\}\cap(b-\eps,b]=\emptyset$,
		\item for every $i\in\N$, and every interval $J\subset I$ of diameter $\geq \eps/2$ it holds that $f^{n_i-n_{i-1}}(J)=I$. 
	\end{enumerate}
\end{lemma}
\begin{proof}
	Let $i\geq 0$, and $j\in\N$. We will first prove that $B(f^{j+1},x_{i+j+1})\subseteq B(f^j,x_{i+j})$. Let $J=J(f^{j+1},x_{i+j+1})$, so $x_{i+j+1}\in J$, and $J$ is maximal such that $f^{j+1}|_J$ is one-to-one. Then also $f|_J$ is one-to-one, and since $f(x_{i+j+1})=x_{i+j}$, it follows that $f(J)\ni x_{i+j}$, and $f^j|_{f(J)}$ is one-to-one. In particular, $B(f^j,x_{i+j})\supseteq f^{j+1}(J)=B(f^{j+1},x_{i+j+1})$.
	
	For every $i\geq 0$ we define $A_i(x)=\bigcap_{j\in\N}B(f^j,x_{i+j})$. Since $f$ is post-critically finite, Remark~\ref{remark:branches} implies that $\{A_i(x):i\geq 0\}$ is finite. The first part of the proof also implies that for every $i\geq 0$ there is $J(i)$ such that $B(f^j,x_{i+j})=A_i(x)$, for every $j\geq J(i)$. 
	
	Since $\{A_i:i\geq 0\}$ is finite, we can find a strictly increasing sequence $(m_i)_{i\geq 0}$ and $a<b\in I$ such that $A_{m_i}(x)=[a,b]$ for all $i\geq 0$. Now we define the strictly increasing sequence $(m'_i)_{i\geq 0}$ as $m'_0:=m_0$, and $m'_{i}:=\min\{m_j:  m_j-m'_{i-1}\geq J(m'_{i-1})\}$, for $i>0$. So, since $m'_i-m'_{i-1}\geq J(m'_{i-1})$, we have $B(f^{m'_i-m'_{i-1}},x_{m'_i})=B(f^{m'_i-m'_{i-1}},x_{m'_{i-1}+(m'_i-m'_{i-1})})=A_{m'_{i-1}}(x)=[a,b]$, for every $i\geq 1$.
	
	Note that actually $B(f^{m'_i-m'_{i-1}+k},x_{m'_i+k})=[a,b]$, for every $i\in\N$ and every $k\geq 0$. Furthermore, since $B(f^{m'_i-m'_{i-1}},x_{m'_i})=[a,b]$, and since $f^{m'_i-m'_{i-1}}(x_{m'_i})=x_{m'_{i-1}}$, it follows that $x_{m'_i}\in[a,b]$ for every $i\in\N$. In particular, there is a strictly increasing subsequence $(n'_i)_{i\geq 0}\subset (m'_i)_{i\geq 0}$ such that $(x_{n'_i})_{i\geq 0}$ converges to $y\in [a,b]$. Thus we can assume that there is $\eps>0$ such that $[a,a+\eps)\cap\{x_{n'_i}:i\geq 0\}=\emptyset$, or $\{x_{n'_i}:i\geq 0\}\cap (b-\eps,b]=\emptyset$. 
	
	Since $f$ is leo, there is $N\in\N$ such that for every interval $J\subset I$ of diameter $\geq \eps/2$ it holds that $f^n(J)=I$ for every $n\geq N$. We find a strictly increasing subsequence $(n_i)_{i\geq 0}\subset(n'_i)_{i\geq 0}$ such that $n_i-n_{i-1}>N$ for every $i\in\N$. Then it also holds that $[a,a+\eps)\cap\{x_{n_i}:i\geq 0\}=\emptyset$, or $\{x_{n_i}:i\geq 0\}\cap (b-\eps,b]=\emptyset$. Thus $(n_i)_{i\geq 0}$ satisfies $(2)$ and $(3)$.
	
	We only have to show that $(n_i)_{i\geq 0}$ satisfies $(1)$. Let $j> 0$. Then there are $i_1>i_2$ such that $n_j=m'_{i_1}$, and $n_{j-1}=m'_{i_2}$. Since $B(f^{m'_{i_2+1}-m'_{i_2}+k},x_{m'_{i_2+1}+k})=[a,b]$ for every $k\in\N$, by taking $k=m'_{i_1}-m'_{i_2+1}\geq 0$, we get $B(f^{m'_{i_1}-m'_{i_2}},x_{m'_{i_1}})=B(f^{n_j-n_{j-1}},x_{n_j})=[a,b]$, which finishes the proof.
\end{proof}

\begin{corollary}\label{cor:main}
	Let $f\colon I\to I$ be a piecewise monotone, post-critically finite, leo map, and let $x\in\underleftarrow{\lim}(I,f)=X_f$. Then there exists a thin planar embedding $\nu_x\colon X_f\to\R^2$ such that $\nu_x(x)$ is an accessible point of $\nu_x(X_f)$.
\end{corollary}
\begin{proof}
	Let $(n_i)_{i\geq 0}$ be as in Lemma~\ref{lem:postcritfin}. Then Theorem~\ref{thm:main} implies that there is a thin planar embedding $\nu'_x\colon\underleftarrow{\lim}(I,f^{n_i-n_{i-1}})\to\R^2$ such that $\nu'_x((x_{n_0},x_{n_1},x_{n_2},\ldots))$ is an accessible point of $\nu'_x(\underleftarrow{\lim}(I,f^{n_i-n_{i-1}}))$. Denote by $h\colon X_f\to \underleftarrow{\lim}(I,f^{n_i-n_{i-1}})$ the homeomorphism given by $h((\xi_0,\xi_1,\xi_2,\ldots))=(\xi_{n_0},\xi_{n_1},\xi_{n_2},\ldots)$. Then $h(x)=(x_{n_0},x_{n_1},x_{n_2},\ldots)$, and $\nu_x:=\nu'_x\circ h\colon X_f\to\R^2$ is a thin planar embedding such that $\nu_x(x)$ is an accessible point of $\nu_x(X_f)$. 
\end{proof}

\end{document}